\documentclass[11pt]{article}
\usepackage{amsmath,amssymb,amscd,amsthm}
\usepackage{currfile}
\usepackage{color}
\usepackage{mathrsfs}
\usepackage{graphicx}
\usepackage{caption}
\usepackage{subfigure}
\usepackage{setspace}
\usepackage{epstopdf}

\newtheorem{thm}{Theorem}[section]

\newtheorem{lem}[thm]{Lemma}
\newtheorem{prob}[thm]{Problem}

\def\marker{\>\hbox{${\vcenter{\vbox{
    \hrule height 0.4pt\hbox{\vrule width 0.4pt height 6pt
    \kern6pt\vrule width 0.4pt}\hrule height 0.4pt}}}$}\>}

\newcommand{\mc}{{\rm mader}_{\vec{\chi}}}
\newcommand{\ve}{\vec{\chi}}

\def\lrC{\overset{\textup{\hspace{0.05cm}\tiny$\leftrightarrow$}}{\phantom{\in}}\hspace{-0.32cm}C}
\def\lrH{\overset{\textup{\hspace{0.05cm}\tiny$\leftrightarrow$}}{\phantom{\in}}\hspace{-0.32cm}H}
\def\lrK{\overset{\textup{\hspace{0.05cm}\tiny$\leftrightarrow$}}{\phantom{\in}}\hspace{-0.32cm}K}
\def\lrP{\overset{\textup{\hspace{0.05cm}\tiny$\leftrightarrow$}}{\phantom{\in}}\hspace{-0.32cm}P}
\def\lrT{\overset{\textup{\hspace{0.05cm}\tiny$\leftrightarrow$}}{\phantom{\in}}\hspace{-0.32cm}T}

\begin{document}
\title{Some Mader-perfect graph classes}
\author{
 Hui Lei$^1$, Siyan Li$^2$, Xiaopan Lian$^2$, Susu Wang$^2$\\
\small $^1$ School of Mathematical Sciences and LPMC, Nankai University, Tianjin 300071, China\\
\small $^2$ Center for Combinatorics and LPMC, Nankai University,  Tianjin 300071, China\\
\small Emails: hlei@nankai.edu.cn; 2120200005@mail.nankai.edu.cn; \\ \small xiaopanlian@mail.nankai.edu.cn; zzushuxuewangsusu@163.com
}

\maketitle
\begin{abstract}
The  dichromatic number of $D$, denoted by $\overrightarrow{\chi}(D)$, is the smallest integer $k$ such that $D$ admits an acyclic $k$-coloring. We use $\mc(F)$ to denote the smallest integer $k$ such that if $\overrightarrow{\chi}(D)\ge k$, then $D$ contains a subdivision of $F$. A digraph $F$ is called  Mader-perfect if for every subdigraph $F'$ of $F$, ${\rm mader }_{\overrightarrow{\chi}}(F')=|V(F')|$. We extend octi digraphs to a larger class of digraphs and prove that it is Mader-perfect, which generalizes a result of Gishboliner, Steiner and Szab\'{o} [Dichromatic number and forced subdivisions,  {\it J. Comb. Theory, Ser. B} {\bf 153} (2022) 1--30].  We also show that if $K$ is a proper subdigraph of $\lrC_4$ except for  the digraph obtained from $\lrC_4$ by deleting an arbitrary arc, then $K$ is Mader-perfect.

\noindent\textbf{Keywords:} Digraph; Dichromatic number; Subdivision; Strongly connected;

\end{abstract}

\section{Introduction}
Let $\gamma$ be a (di)graph parameter. Given a (di)graph $H$, we use {\it ${\rm mader}_\gamma(H)$} to denote the smallest integer such that  if $\gamma(G)\ge {\rm mader}_\gamma(H)$, then $G$ contains a subdivision of $H$. There are many results and open problems about the relation between some graph parameters and the containment of topological minor.
For example, let $\gamma$ be the chromatic number $\chi$.  Haj\'{o}s' conjecture now can be restated as ${\rm mader}_\chi(K_k)=k$. Haj\'{o}s' conjecture holds for $k\le 4$ by Dirac in \cite{GAD1952} and it was disproved for $k\ge 7$ by Catlin in \cite{PAC}.
For the cases $k = 5$ and $k=6$, Haj\'{o}s conjecture is still open. Now let $\gamma$ be the minimum degree $\delta$. For $k\le 4$, ${\rm mader}_\delta(K_k)=k-1$ which was proved by Dirac in \cite{GAD1960}. For general case, the following was proved.
\begin{thm}\label{cmd}\cite{BA,JE}
There exists an absolute constant $C>0$ such that for every integer $k\in \mathbb{N}$, ${\rm mader}_\delta(K_k)\le Ck^2$.
\end{thm}

Since every graph $G$ with chromatic number $\chi(G)$ contains a subgraph which has minimum degree at least $\chi(G)-1$, we can deduce from Theorem~\ref{cmd} that ${\rm mader}_\chi(K_k)\le Ck^2$. Therefore, if $\chi(G)$ is sufficiently large, then $G$ contains a subdivision of any given graph.

Motivated by the above discussion, Aboulker, Cohen, Havet, Lochet, Moura, and Thomass\'{e} \cite{PNF} generalized this kind of research to digraphs. Given a digraph $D$, an {\it acyclic $k$-coloring} of $D$ can be defined as a mapping $c: V(D)\rightarrow \{1,\ldots,k\}$ such that the {\it color class} $c^{-1}(i)$ induces an acyclic subdigraph of $D$
 for each $i\in \{1,\ldots,k\}$. The {\it dichromatic number} of $D$, denoted by $\overrightarrow{\chi}(D)$, is the smallest integer $k$ such that $D$ admits an acyclic $k$-coloring.
The authors in \cite{PNF} proved the following result.
\begin{thm}\cite{PNF}
Let $F$ be a digraph with $n$ vertices and $m$ arcs. Then $\mc(F)\le 4^m(n-1)+1$.
\end{thm}

We use $(u,v)$ to denote an arc with head $v$ and tail $u$. For a graph $H$, its \emph{biorientation} is the digraph $\lrH$ obtained by replacing each edge $uv$ of $H$ by the pair of symmetric arcs $(u,v)$ and $(v,u)$, and an {\it orientation} of $H$ is a digraph obtained from $H$ by replacing each edge $uv$ in $H$ by either $(u,v)$ or $(v,u)$. Later the bound on $\mc(\lrK_k)$ was modified by Goshboliner, Steiner and Szab\'{o} \cite{LRT}.

\begin{thm} \cite{LRT} For every $k\in \mathbb{N}$,
$$\Omega(\frac{k^2}{\log k})\le \mc(\lrK_k) \le 4^{k^2-2k+1}(k-1)+1.$$
\end{thm}

Since there is no analogue of Theorem~\ref{cmd} in digraphs, it seems quite challenging to obtain a polynomial bound. Goshboliner, Steiner and Szab\'{o} \cite{LRT} found better bounds for special digraphs classes. One of the digraph classes is called {\it octus}, denoted by $\mathscr{F}^*$, and they defined it as follows.
\begin{itemize}
\item $K_1\in \mathscr{F}^*$.

\item Let $F\in\mathscr{F}^*$ with $v_0\in V(F)$. Let $P=v_1,\ldots, v_k, k\ge1$ be an orientation of a path which is disjoint from $V(F)$. Let $F^*$ be obtained from $F$ by adding the path $P$, both arcs $(v_0, v_1)$, $(v_1, v_0)$, and (if $k \ge 2$) exactly one of the arcs $(v_0, v_k)$, $(v_k, v_0)$. Then $F^*\in\mathscr{F}^*$.

\item If $F\in\mathscr{F}^*$, then every subdigraph of $F$ is also in $\mathscr{F}^*$.
\end{itemize}
\noindent They established the following Theorem.

\begin{thm}\cite{LRT}\label{FF*}
For every $F\in \mathscr{F}^*$, we have $\mc(F) = |V(F)|$.
\end{thm}

A digraph $F$ is called {\it Mader-perfect} if for every subgraph $F'$ of $F$, ${\rm mader }_{\overrightarrow{\chi}}(F')=|V(F')|$. Note that for any digraph $F$, we have $\mc(F)\ge |V(F)|$ since $\ve(\lrK_{|V(F)|-1})=|V(F)|-1$ but $\lrK_{|V(F)|-1}$ does not contain a subdivision of $F$. Gishboliner, Steiner and Szab\'{o} \cite{LRT} showed that all the digraphs in $\mathscr{F}^*$ are Mader-perfect and they thought it would be interesting to characterize Mader-perfect digraph classes. Let $C_\ell$ be a cycle of length $\ell$. The authors in \cite{LRT} proposed the following problems.

\begin{prob}\cite{LRT}
Characterize Mader-perfect digraphs.
\end{prob}

\begin{prob}\cite{LRT}
Determine $\mc(\lrC_\ell)$.
\end{prob}

As an attempt to solve the two problems, we generalize Theorem~\ref{FF*} to a larger digraph family $\mathscr{F}$. In order to establish the definition of $\mathscr{F}$, we call a subdigraph $Q$ of the biorientation of $P_{k}=v_1v_2\ldots v_k$ is {\it good} if
\begin{itemize}

\item[(a)] $Q$ is connected with $|V(Q)|=k$ and all the digons(the pair of symmetric arcs) in $Q$ are vertex disjoint;

\item[(b)] $d_Q^+(v_1)+d_Q^-(v_1)=1$.

\end{itemize}

\noindent Denote the set of all the good subdigraphs of $\lrP_{k}$ by $\mathscr{S}_k$. Let $\mathscr{F}$ be a family of digraphs,  which  can be recursively defined as follows.
\begin{itemize}
\item $K_1\in \mathscr{F}$.

\item Let $F\in \mathscr{F}$ with $v_0\in V(F)$ and $Q\in \mathscr{S}_k$. Let $F^*$ be obtained from $F$ by adding $Q$, both arcs $(v_0, v_1)$, $(v_1, v_0)$, and (if $k\ge 2$) exactly one of the arcs $(v_0, v_k)$, $(v_k, v_0)$. Then $F^*\in \mathscr{F}$.

\item If $F\in \mathscr{F}$, then every subdigraph of $F$ also belongs to $\mathscr{F}$.

\end{itemize}
Our main result is as follows:
\begin{thm}\label{FF}
For every $F\in \mathscr{F}$, we have $\mc(F) = |V(F)|$.
\end{thm}
We also show that the  if $K$ is a proper subdigraph of $\lrC_4$ except for  the digraph obtained from $\lrC_4$ by deleting an arbitrary arc, then $\mc(K)=|V(K)|$.

The paper is organized as follows. In section 2, we list some notions and useful results. In section 3, we show that each digraph contained in the generalized family $\mathscr{F}$ is Mader-perfect. Especially, the subdigraphs of $\lrC_\ell$ where the digons are vertex disjoint are contained in $\mathscr{F}$. In section 4, we prove that the proper subdigraphs of $\lrC_4$ except for  $H_3$ (as shown in Figure~\ref{C4--}) are Mader-perfect.

\section{Preliminary}
In this section, we first introduce some notions. For a digraph $D$ and $u\in V(D)$, let $\delta^+_D(u)$ be the number of arcs in $D$ with tail $u$  and  $\delta^-_D(u)$ be the number of arcs in $D$ with head $u$. Let $\delta^+(D)=\min\{\delta^+_D(u):u\in V(D)\}$ and $\delta^-(D)=\min\{\delta^-_D(u):u\in V(D)\}$. For a vertex subset $X$ in $D$, the subdigraph of $D$ induced by $X$ is denoted by $D[X]$.
We say that $W=u_0e_1u_1\ldots u_{k-1}e_ku_k$ is a {\it diwalk} from $u_0$ to $u_k$ in $D$ if $u_i\in V(D)$ for $i\in \{0,1,\ldots,k\}$ and $e_j=(u_{j-1},u_j)\in A(D)$ for $j\in \{1,\ldots,k\}$. Further, if $u_i\neq u_j$ for $i,j\in \{0,1,\ldots,k\}$, then we call $W$ a {\it dipath} from $u_0$ to $u_k$. If only $u_0=u_k$, then we call  $W$ a {\it dicycle} of length $k$.
A digraph $D$ is called {\it strongly connected} if there is a dipath from $u$ to $v$ for every pair of vertices $u$ and $v$ in $D$. A {\it strong component} of $D$ is a maximal induced subgraph of $D$ which is strongly connected.  And $D$ is called {\it $k$-strongly connected} if the subdigraph $D[V(D)\setminus S]$ is strongly connected for any vertex subset $S\subseteq V(D)$ with $|S|<k$. We say that $D$ is {\it $k$-dicritical}, if $\ve(D)=k$ and for any proper subdigraph $D'\subseteq D$, $\ve(D')<k$. For $X,Y\subseteq V(D)$, an $(X,Y)$-dipath is a dipath  which starts in a vertex of $X$, ends in a vertex of $Y$, and is
internally vertex-disjoint from $X\cup Y$, especially, if $X=\{v\}$, then we write $(v,Y)$-dipath for short. For $u,v\in V(D)$, $dist_D(u,v)$ is defined as the length of the shortest dipath between $u$ and $v$ contained in $D$. We use $D[u,v]$ to denote a directed path from $u$ to $v$ and $D(u,v)$ to denote a directed path between $u$ and $v$ in $D$. Especially, $D[u,v]$ and $D(u,v)$ denote the vertex $u$ when $u=v$.
Let $\overleftarrow{D}$ be the digraph obtained from $D$ by reversing the  orientations of all arcs.

The following results were proved in \cite{JG,LRT,WM} and will be used in this paper.
\begin{lem}\cite{LRT}\label{DrD}
Let $F$ be a digraph, we have $\mc(F)=\mc(\overleftarrow{F})$.
\end{lem}

\begin{lem}\cite{LRT}\label{disu}
Let $F$ be the disjoint union of two digraphs $F_1$ and $F_2$, then $\mc(F)\le \mc(F_1)+\mc(F_2)$.
\end{lem}

\begin{lem}\cite{LRT}\label{Dici}
Let $D$ be $k$-dicritical. Then $\delta^+(D), \delta^-(D)\ge k-1$ and $D$ is strongly connected.
\end{lem}

\begin{thm}\cite{JG,WM}\label{D-v} Let $k\in \mathbb{N}$, and let $D$ be a $k$-strongly connected digraph with $\delta^+(D)$, $\delta^-(D)\ge 2k$. Then there is $x \in V (D)$ such that $D'=D[V(D)\setminus \{x\}]$ is (also) $k$-strongly connected.
\end{thm}

\section{The proof of Theorem \ref{FF}}
To prove Theorem \ref{FF}, we first prove the following result.

\begin{thm}\label{FQ}
Let $F$ be a digraph with $v_0\in V(F)$ and $Q\in \mathscr{S}_k$. Let $F^*$ be obtained from $F$ by adding $Q$, both arcs $(v_0, v_1)$, $(v_1, v_0)$, and (if $k\ge 2$) exactly one of the arcs $(v_0, v_k)$, $(v_k, v_0)$. Then $\mc(F^*)\le \mc(F) + k$.
\end{thm}

\begin{proof} By the proof of Case (II) of Theorem 16  in \cite{LRT}, we only consider the case $(v_0, v_k)\in A(F^*)$. For convenience, let $\mc(F)=M$. We need to show that for any given digraph $D$ with $\ve(D)=M+k$, there is a subdivision of $F^*$ in $D$. Let $c_0: V(D)\rightarrow \{1,\ldots,M+k\}$ be an acyclic coloring which maximizes $|c_0^{-1} (\{1, \ldots, k\})|$. In the rest of the proof, let $D_1=D[c_0^{-1}(\{1,\ldots,k\})]$ and $D_2=D[c_0^{-1}(\{k+1,\ldots,M+k\})]$, we have that $\ve(D_1)=k$ and $\ve(D_2)=M$ since $c_0$ is an acyclic coloring of $D$.
Combining that $\mc(F)=M$, there is a subdivision $S$ of $F$ in $D_2$ and denote by $x_0$ the vertex in $S$ corresponding to $v_0$.
In \cite{LRT}, the authors defined a pre-order on the acyclic colorings of $D_1$ with respect to $x_0$ as follows.
For each acyclic $k$-coloring $c$ of $D_1$,  define a vector $\mathbf{v}(c)\in \mathbb{Z}^{k}$ with $\mathbf{v}(c)_i=|N^+_D(x_0)\cap c^{-1}(i)|$ for $i\in [k]$.
Now, consider the pre-order on the set of acyclic $k$-colorings of $D_1$, where $c_1\prec c_2$ iff $\mathbf{v}(c_1) <_{\mathrm{lex}} \mathbf{v}(c_2)$. Here $<_{\mathrm{lex}}$ denotes the lexicographical order on $\mathbb{Z}^{k}$. Let $c$ denote an acyclic $k$-coloring of $D_1$ that is minimal with respect to $\prec$. From {\bf Claim 2} in the proof of Theorem 16 in \cite{LRT}, we know that there are vertices $x_1,\ldots,x_k$ in $N^+(x_0)\cap V(D_1)$ such that
\begin{itemize}
\item $c(x_i)=i$ for $i\in \{1, \ldots, k\}$;

\item there is a dicycle $\hat{C}$ in $D$ containing $x_0$ and $x_1$ such that $V(\hat{C})\setminus \{x_0\}\subseteq c^{-1}(1)$;
\item there is a strong component of $D[c^{-1}(\{i-1,i\})]$ that contains both $x_{i-1}$ and $x_{i}$ for $i\in \{2, \ldots, k\}$.
\end{itemize}

For $2\leq i\leq k$, let $X_{i-1,i}$ be the strong component of $D[c^{-1}(\{i-1,i\})]$ that contains $x_{i-1}$ and $x_i$. If $\{(v_{i-1}, v_i),(v_i, v_{i-1})\}\not\subseteq E(Q)$, then we choose a directed path $P_{i-1,i}$ in
$X_{i,{i+1}}$ such that $P_{i-1,i}$ is directed from $x_{i-1}$ to $x_i$ if $(v_{i-1}, v_i)\in A(Q)$ and from $x_i$ to $x_{i-1}$
if $(v_{i}, v_{i-1})\in A(Q)$. Note that $\{(v_1, v_2),(v_2, v_1)\}\not\subseteq E(Q)$ since $d_Q^+(v_1)+d_Q^-(v_1)=1$.
Next, we show that there exist vertices $z_1,z_2,\ldots,z_k$ in $D_1$ which satisfy
\begin{itemize}
\item $z_1\in V(\hat{C})$.
\item for every $2\le i\le k$, there exists a dipath $Q_{i-1,i}$ in $D_1$ from $z_{i-1}$ to $z_i$ if $(v_{i-1},v_i)\in A(Q)$, and exists a dipath $Q_{i,i-1}$ in $D_1$
from $z_i$ to $z_{i-1}$ if $(v_{i},v_{i-1})\in A(Q)$.
\item the dipaths $Q_{i-1,i}$ and $Q_{j,j-1}$, $i,j \in\{2,\ldots, k\}$, are pairwise internally vertex-disjoint.
\item either $V(\hat{C})\cap V(Q_{1,2})=z_1$ or $V(\hat{C})\cap V(Q_{2,1})=z_1$.
\item   $V(\hat{C})\cap V(Q_{i-1,i}) =\emptyset$ and $V(\hat{C})\cap V(Q_{j,j-1}) =\emptyset$ for $i,j\in \{3,\ldots,k\}$.

\end{itemize}

Define $z_1\in V(\hat{C})$ to be the unique last vertex in $V(\hat{C})$ that we meet when traversing the trace of the path $P_{1,2}$ starting from $x_1\in V(\hat{C})$.  Without loss of generality, we assume that $(v_1,v_2)\in A(Q)$. Note that $P_{1,2}[z_1, x_2]$ is a dipath from $z_1$ to $x_2$. Now, we determine $z_2$ and $z_3$. Suppose that $\{(v_2, v_3),(v_3, v_2)\}\not\subseteq A(Q)$. Then for $j\in\{2,3\}$, we define $z_j$ to be the first vertex of
$P_{j,j+1}$ that we meet when traversing the trace of the dipath $P_{j-1,j}(z_{j-1}, x_j)$ starting from $z_{j-1}$, and let $Q_{i-1,i}=P_{i-1,i}(z_{i-1}, z_i)$ for $i\in\{2,3\}$.

Suppose that $\{(v_2, v_3),(v_3, v_2)\}\subseteq A(Q)$. Note that $\{(v_3, v_4),(v_4, v_3)\}\not\subseteq A(Q)$ since all the digons in $Q$ are vertex disjoint.
If there is a dicycle $C$ in $X_{2,3}$ such that  $V(P_{1,2}[z_1, x_2])\cap V(C)\neq\emptyset$ and $V(P_{3,4})\cap V(C)\neq\emptyset$, then let $z_2\in V(C)\cap V(P_{1,2}[z_1, x_2])$ such that $dist_{P_{1,2}}(z_1,z_2)$ is as small as possible and  let $z_3\in V(C)\cap V(P_{3,4})$ such that $dist_{P_{3,4}}(z_3,x_4)$ is as small as possible.
We have that $c(z_i)=i$ for $i\in \{2,3\}$ and $P_{3,4}[z_3, x_4]$ is a dipath between $z_3$ and $x_4$. We define $Q_{1,2}=P_{1,2}(z_1,z_2)$ and $Q_{2,3}\cup Q_{3,2}=C$. Now, assume that for any dicycle $C$ in $X_{2,3}$, we have that $V(P_{1,2}[z_1, x_2])\cap V(C)=\emptyset$ or $V(P_{3,4})\cap V(C)=\emptyset$. Since $X_{2,3}$ is a strong component, any vertex in $X_{2,3}$ lies in a dicycle. Let $C_1$ and $C_2$ be two dicyles in $X_{2,3}$ which contain $x_2$ and $x_3$ respectively.  We have that $C_1\neq C_2$, $V(P_{1,2}[z_1, x_2])\cap V(C_2)=\emptyset$ and $V(P_{3,4})\cap V(C_1)=\emptyset$ by the  assumption that  $V(P_{1,2}[z_1, x_2])\cap V(C)=\emptyset$ or $V(P_{3,4})\cap V(C)=\emptyset$ for any dicycle $C$ in $X_{2,3}$.

{\bf Case 1} $V(C_1)\cap V(C_2)\neq\emptyset$.

In this case, there are two vertices $u$ and $v$ (it may happen that $u=v$)  in $V(C_1)\cap V(C_2)$ satisfying  $C_2[x_3, u]$, $C_2[v, x_3]$, and $C_1$ are pairwise internally vertex disjoint.
Let $z_2\in V(C_1)\cap V(P_{1,2}[z_1, x_2])$ such that $dist_{P_{1,2}}(z_1,z_2)$ is as small as possible. Let $z'_3\in V(C_2[x_3,u]\cup C_2[v,x_3])\cap V(P_{3,4})$ such that $dist_{P_{3,4}}(z'_3,x_4)$ is as small as possible. If $(v_3,v_4)\in E(Q)$, then let $z_3=v$ and $C_2[v,z'_3]\cup P_{3,4}[z'_3, x_4]$ is a dipath from $z_3$ to $x_4$. If $(v_4,v_3)\in E(Q)$, then let $z_3=u$ and $P_{3,4}[x_4, z'_3]\cup C_2[z'_3, u]$ is a dipath from $x_4$ to $z_3$.  We define $Q_{1,2}=P_{1,2}[z_1,z_2]$ and $Q_{2,3}\cup Q_{3,2}=C_1$.

{\bf Case 2} $V(C_1)\cap V(C_2)=\emptyset$.

Since $X_{2,3}$ is a strong component, there is a shortest dipath $P_1$ from $C_1$ to $C_2$ and a shortest dipath $P_2$ from $C_2$ to $C_1$. Denote the initial and terminal vertices of $P_1$ by $u$ and $v$, the initial and terminal vertices of $P_2$ by $x$ and $y$, respectively.
Let $z'_3\in V(P_1\cup C_2)\cap V(P_{3,4})$ such that $dist_{P_{3,4}}(z'_3,x_4)$ is as small as possible. Now we need to consider the following cases.

{\bf Case 2.1} $z'_3\in V(C_2)$.

Denote by $w$ the unique first vertex in $V(P_1\cup C_1)\setminus\{v\}$ that we meet when traversing the trace of the dipath $P_2$ starting from $x$. We have that $P_2[x,w]\cup P_1[w,v]\cup C_2[v,x]$ is a dicycle when $w\neq y$ and $P_2\cup C_1[y,u]\cup P_1\cup C_2[v,x]$ is a dicycle when $w=y$.
Note that when $w=y$, we may assume $x_2\notin V(C_1[y,u])$ and $x_3\notin V(C_2[v,x])$,  otherwise it can be reduced to  Case 1.
Denote the above dicycle by $C_3$. Let $z'_2\in V(C_1\cup P_1\cup P_2[x,w])\cap V(P_{1,2}[z_1,x_2])$ such that $dist_{P_{1,2}}(z_1,z'_2)$ is as small as possible. Note that $z'_2\neq z'_3$ since $c(z'_2)=2$ and $c(z'_3)=3$.

Suppose that $w\neq y$.
If $z'_2\in V(C_1)$, then let $z_2=w$. We define $Q_{1,2}=P_{1,2}[z_1,z'_2]\cup C_1[z'_2,u]\cup P_1[u,w]$ and $Q_{2,3}\cup Q_{3,2}=C_3$.
If $z'_2\in V(P_1[u,w])$, then let $z_2=w$. We define $Q_{1,2}=P_{1,2}[z_1,z'_2]\cup P_1[z'_2,w]$ and $Q_{2,3}\cup Q_{3,2}=C_3$.
If $z'_2\in V(P_1[w,v]\cup P_2[x,w])$, then let $z_2=z'_2$.  We define $Q_{1,2}=P_{1,2}[z_1,z_2]$ and $Q_{2,3}\cup Q_{3,2}=C_3$.
Suppose that $w=y$. If $z'_2\in V(C_1)\setminus V(C_3)$, then let $z_2=w$. We define $Q_{1,2}=P_{1,2}[z_1,z'_2]\cup C_1[z'_2,w]$ and $Q_{2,3}\cup Q_{3,2}=C_3$.
If $z'_2\in V(C_3)$, then let $z_2=z'_2$ and so $c(z_2)=2$. We define $Q_{1,2}=P_{1,2}[z_1,z_2]$ and $Q_{2,3}\cup Q_{3,2}=C_3$.

%

Now, we define $z_3$. Suppose that $V(P_{3,4}(x_4,z'_3))\cap V(P_2[x,w]\cup C_2[v,x])=\emptyset$. If $(v_3,v_4)\in A(Q)$, then let $z_3=x$ and  $C_2[x, z'_3]\cup P_{3,4}[z'_3, x_4]$ is a dipath from $z_3$ to $x_4$.
If $(v_4,v_3)\in A(Q)$, then let $z_3=v$ and $P_{3,4}[x_4, z'_3]\cup C_2[z'_3, v]$ is a dipath from $x_4$ to $z_3$. Suppose that $V(P_{3,4}(x_4,z'_3))\cap V(P_2[x,w]\cup C_2[v,x])\neq\emptyset$, then let $z_3\in V(P_{3,4}(x_4,z'_3))\cap V(P_2[x,w]\cup C_2[v,x])$ such that $dist_{P_{3,4}}(z_3,x_4)$ is as small as possible and $P_{3,4}(z_3,x_4)$ is a dipath between $z_3$ and $x_4$. By the definition of $z'_3$, we have that $z_3\neq w$ and $c(z_3)=3$.

{\bf Case 2.2} $z'_3\in V(P_1)\setminus V(C_1)$.

If $w\neq y$, then one of the following holds.
 \begin{itemize}
 \item[(a)] There is a vertex $u_1\in V(P_1)\cap V(P_2)$ with $z'_3\in V(P_1[u,u_1])$ such that $P_2[u_1,y]$ is internally vertex disjoint with $P_1[u,u_1]$.
\item[(b)] There is a vertex $u_1\in V(P_1)\cap V(P_2)$ with $z'_3\in V(P_1[u_1,v])\setminus\{u_1\}$ such that $P_2[x,u_1]$ is internally vertex disjoint with $P_1[u_1,v]$.
\item[(c)] There are two distinct vertices $u_1,u_2\in V(P_1)\cap V(P_2)$ with $u_1\neq z'_3$ such that $P_1(u_1,u_2)$ is internally vertex disjoint with $P_2(u_1,u_2)$ and $z'_3\in V(P_1(u_1,u_2))$.
\end{itemize}

If $w=y$ or (a) holds, then let $C_1=C_1$ and $C_2=C_1[y,u]\cup P_1[u,u_1]\cup P_2[u_1,y]$. We refer $z'_3$ as $x_3$ and so it can be reduced to Case 1. Suppose that (b) or (c) holds, we have $P_2[x,u_1]\cup P_1[u_1,v]\cup C_2[v,x]$ is a dicycle or $P_1(u_1,u_2)\cup P_2(u_1,u_2)$ is a dicycle, respectively. Call the possible cycle $C_3$.  Let $C_1=C_1$ and $C_2=C_3$. We refer $z'_3$ as $x_3$ and so it can be reduced to  Case 2.1.

It is easy to check that $Q_{1,2}$, $Q_{2,3}$, $Q_{2,3}$, and $V(P_{3,4}(z_3, x_4))$  are pairwise internally vertex-disjoint.  Therefore, recursively, we can find the corresponding $z_4,\ldots,z_k$ with respect to $v_4,\ldots,v_k$.

Let $S^*$ be the subdigraph of $D$ formed by joining $S\subseteq D[Y_2]$, the pairwise distinct
vertices $z_1,\ldots,z_k$ and the connecting dipaths $Q_{i-1,i}$, $Q_{i,i-1}$, $i \in\{ 2,\ldots, k\}$, the two anti-parallel
directed paths $\hat{C}[x_0, z_1]$, $\hat{C}[z_1, x_0] $between $x_0$ and $z_1$ as well as the arc $(x_0, z_k)$. It follows
that $S^*$ is isomorphic to a subdivision of $F^*$.

This completes the proof of Theorem \ref{FQ}.
\end{proof}





By the definition of $\mathscr{F}$, we know that $\mathscr{F}^*\subseteq \mathscr{F}$. We call the second item the ear addition operation and the third item taking a subdigraph operation in the definition of $\mathscr{F}$. Therefore, every digraph in $\mathscr{F}$ can be obtained from $K_1$ by a sequence of operations consisting of the above two operations. We call  $F\in \mathscr{F}$  {\it maximal} if it is obtained from $K_1$ by only using the ear addition operation. By recursively using Theorem~\ref{FQ}, for $F\in \mathscr{F}$ being maximal, $\mc(F)=|V(F)|$.

{\bf Proof of Theorem \ref{FF}}.  Let $F\in \mathscr{F}$ and let $UG(F)$ be the graph obtained from $F$ by replacing every arc $(u, v)$ with the
edge $uv$ and deleting all multiple edges between every pair of vertices apart from one.  If $UG(F)$ is disconnected, then by Lemma~\ref{disu}, it suffices to consider the components of $UG(F)$. Therefore, in the following, we may assume that $UG(F)$ is connected.  We prove Theorem~\ref{FF} by induction on the number of  cycles contained in $UG(F)$. Denote by $\ell$ the number of  cycles contained in $UG(F)$. First, assume that $\ell=0$. Then $F\in \mathscr{F}^*$ and so $\mc (F)=|V(F)|$ by Theorem~\ref{FF*}.

Now, suppose that $\ell \ge 1$. We call a  cycle $C$ an outmost cycle if there is a vertex $u\in V(C)$ such that the other vertices of $C$ cannot reach any other  cycle in $UG(F)$ without passing $u$. Call $u$ the special vertex of $C$. Let $C$ be an outmost  cycle in $UG(F)$  with vertex set $\{v_0,\ldots, v_s\}$ and special vertex $v_0$.
Let $Y$ be the component of $F[V(F)\setminus \{v_1,\ldots, v_s\}]$ that contains $v_0$. By the induction hypothesis, $\mc(Y)=|V(Y)|$. Let $Y_1=F[V(Y)\cup \{v_1,\ldots, v_s\}]$. Therefore, by Theorem~\ref{FQ},  $\mc(Y_1)=|V(Y_1)|$ since $Y_1$ can be obtained from $Y$ and $F[\{v_1,\ldots,v_s\}]$ by the ear addition operation. For each $v_i$ with $i\in \{1,\ldots,s\}$, Let $F_{v_i}$ be the digraph that contains $v_i$ obtained from $F$ by deleting the arcs adjacent to $v_i$ contained in $F[\{v_1,\ldots,v_s\}]$. Since $C$ is an outmost cycle,
$F_{v_i}$ is a spanning subdigraph of a biorientation of some tree rooted at ${v_i}$, say $\lrT_{v_i}$. Note that $\lrT_{v_i}$ can be seen as a maximal digraph of $\mathscr{F}$ obtained from ${v_i}$ by a sequence of ear addition operation (in the case, $P_k$ is an isolate vertex in the definition of $\mathscr{F} $) for $i\in \{1,\ldots,s\}$. Let $T=\lrT_{v_1}\cup\cdots \cup \lrT_{v_s}$.
Denote by $\hat{F}$ the digraph obtained from $Y_1$ and $V(T)\setminus\{v_1,\ldots,v_s\}$ by a corresponding sequence of ear addition operation  of each $\lrT_{v_i}$,  where $i\in \{1,\ldots,s\}$.  Thus, $\mc(\hat{F})=|V(\hat{F})|$ by recursively using Theorem~\ref{FQ}.
Since $F$ is a spanning subdigraph of $\hat{F}$,  $\mc(F)\le \mc(\hat{F})$. Therefore, $\mc(F)=|V(F)|$ since $\mc(F)\ge|V(F)|$.
$\square$

\section{Subdigraphs of the biorientation of $C_4$}
Let $\mathcal{K}=\{K| K$ be a proper subdigraph of $\lrC_4$ and $K\not\cong H_3$\}, where $H_3$ is as shown in Figure~\ref{C4--}. In this section, we show that any graph in $\mathcal{K}$  is mader perfect. For any $K\in \mathcal{K}$, we have $K\in \mathscr{F}$ when $K\notin \{H_1, \overleftarrow{H}_1, H_2,\overleftarrow{H}_2\}$, where $H_1$ and $H_2$ are as shown in Figure~\ref{C4--}. To prove that any graph in $\mathcal{K}$  is mader perfect, by Theorem \ref{FF} and Lemma~\ref{DrD}, it suffices to prove that if $K\in \{H_1,H_2\}$, then $\mc(K)=4$.

\begin{figure}[htbp]
  \centering
    \includegraphics[width=12cm]{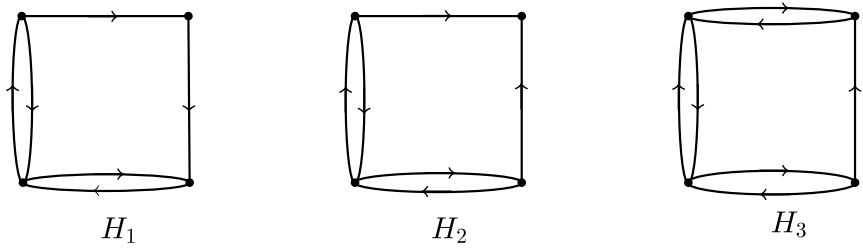}\\
  \caption{Three subdigraphs of the biorienitation of $C_4$}\label{C4--}
\end{figure}

\begin{thm}\label{SH1=4}
Let  $K\in \{H_1,H_2\}$, then $\mc(K)=4$.
\end{thm}

\begin{proof}
 Let $K\in \{H_1,H_2\}$.  Since $\mc(K)\ge|V(K)|$, $\mc(K)\geq4$. We prove Theorem \ref{SH1=4} by contradiction. Suppose that the assertion
is false, i.e., $\mc(K)>4$. Then there is a digraph $D$ with $\ve(D)=4$ that does not contain a subdivision of $K$.  Among all
counterexamples we choose $D$ so that first $|V(D)|$ is minimum,  and then $|A(D)|$ is minimum. By the choice of $D$, $D$ is connected, 4-dicritical. By Lemma~\ref{Dici} and Theorem~\ref{D-v}, we know that $D$ is strongly connected and there is a vertex $x$ in $D$ such that $D'=D[V(D)\setminus\{x\}]$ is strongly connected. Since $D$ is 4-dicritical,
$\ve(D')=3$. In the following let $c=(V_1,V_2,V_3)$ be an acyclic 3-coloring of $D'$,  where $V_1$, $V_2$ and $V_3$ are the color classes. In order to show that there is an acyclic 3-coloring of $D$, we define the following vertex partition of $D'$ with respect to $x$.
\begin{figure}[thbp]
  \centering
  \includegraphics[width=8cm]{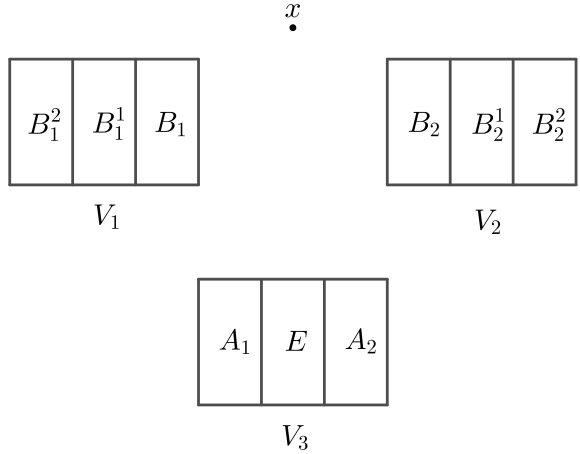}\\
  \caption{The vertex partition of $D'$.}\label{vpod}
\end{figure}

\begin{itemize}
\item[(a)] $B_i\subseteq V_i$ consists of all the vertices that are contained in the same strong component with $x$ in $D[\{x\}\cup V_i]$ for $i\in \{1,2\}$.

\item[(b)] $B^1_i\subseteq V_i$ and $A_i\subseteq V_3$ consist of all the vertices which can be reached by some vertex of $B_i$ in $D$ without passing through any vertex in  $\{x\}\cup B_{3-i}$ for $i\in \{1,2\}$.

\item[(c)] $B^2_i=V_i\setminus(B^1_i\cup B_i)$  for $i\in \{1,2\}$ and  $E=V_3\setminus(A_1\cup A_2)$.
\end{itemize}

Now, let $V'_1=B_2\cup B^1_2\cup A_1\cup E$, $V'_2=B_1\cup B^1_1\cup A_2$ and $V'_3=\{x\}\cup B^2_2$. In order to construct an acyclic 3-coloring $\tilde{c}$ of $D$, we need to consider the vertices in $B^2_1$.
Denote by $T_0$ the set of all the vertices in $B^2_1$ that are contained in some dicycle in the digraph $D[B^2_1\cup V'_3]$ and let $T_3= B^2_1\setminus T_0$. Further, denote by $T_1$ the set of all the vertices in $T_0$ that are not contained in any dicycle in the digraph $D[T_0\cup E]$ and let $T_2=T_0\setminus T_1$.
An 3-coloring $\tilde{c}=(\tilde{V}_1,\tilde{V}_2, \tilde{V}_3)$ of $D$ can be given by setting $\tilde{V}_i=V'_i\cup T_i$, where $\tilde{V}_i$ is a color class of $\tilde{c}$ for $i\in \{1,2,3\}$. If $\tilde{c}$ is acyclic, then it leads to a contradiction as $\ve(D)=4$. Thus, we finish the proof by showing that $D[\tilde{V}_i]$ is acyclic for $i\in \{1,2,3\}$. In the following, we first list some useful claims.
\vspace{0.2cm}

{\bf Claim 1.} $A_1\cap A_2=\emptyset$ for $K\in \{H_1,H_2\}$.

\noindent{\bf Proof.} Suppose that there is a vertex $v\in A_1\cap A_2$. By the partition, there is a shortest dipath $P_1$ from $B_1$ to $v$ without passing through $\{x\}\cup B_2$ in $D$. Let $u\in V(P_1)\cap B_1$.
If $K=H_1$, then let $P_2$ be a shortest dipath from $v$ to $B_2$ without passing through $\{x\}\cup B_1$ in $D$. If $K=H_2$, then let $P_2$ be a shortest dipath from $B_2$ to $v$ without passing through $\{x\}\cup B_1$ in $D$.
Denote the other end of $P_2$ by $w$, and so  $w\in B_2$. Since $B_1$ and $B_2$ are acyclic, there are dicycles $C_u\subset D[\{x\}\cup B_1]$ and $C_w\subset D[\{x\}\cup B_2]$ such that $\{x,u\}\subseteq V(C_u)$ and $\{x,w\}\subseteq  V(C_w)$. Therefore, $C_u\cup C_w\cup P_1\cup P_2$ contains a subdivision of $K$, a contradiction. Thus, $A_1\cap A_2=\emptyset$ for $K\in \{H_1,H_2\}$. $\blacksquare$
\vspace{0.2cm}

{\bf Claim 2.} $D[V'_i]$ is acyclic for $i\in \{1,2,3\}$ and $K\in \{H_1,H_2\}$.

\noindent{\bf Proof.} By the partition, it is obvious that $D[V'_3]$ is acyclic. Suppose that $D[V'_2]$ contains a dicycle $C$. Then $V(C)\cap (B_1\cup B^1_1)\neq\emptyset$ and $V(C)\cap A_2\neq\emptyset$ and so there is a  dipath $P$ from $B_1$ to $V(C)\cap A_2$ without passing through $\{x\}\cup B_2$. This implies that $(V(C)\cap A_2)\subseteq A_1$, a contradiction.
Suppose that $D[V'_1]$ contains a dicycle $C$. Similarly, $D[V'_1\setminus E]$ is acyclic. Therefore, $V(C)\cap E\neq\emptyset$. Since $V_3$ is acyclic, $V(C)\cap (B_2\cup B^1_2)\neq\emptyset$. Let $u\in B_2\cup B^1_2$, then there is a dipath from $B_2$ to $V(C)\cap E$ without passing through $\{x\}\cup B_1$. This implies that $(V(C)\cap E)\subseteq A_2$, a contradiction. Hence, $D[V'_i]$ is acyclic for $i\in \{1,2,3\}$ and $K\in \{H_1,H_2\}$.$\blacksquare$
\vspace{0.2cm}

{\bf Claim 3.} For each vertex $v$ in $T_0$, there is a dipath from $v$ to $B^2_2$ contained in $D[T_0\cup B^2_2]$.

\noindent{\bf Proof.} Suppose not, then there is a vertex $v\in T_0$ such that there is no dipath from $v$ to $B^2_2$ contained in $D[T_0\cup B^2_2]$. Since $D[T_0\cup \{x\}]$ is acyclic and $v$ is contained in a dicycle $C$ in $D[T_0\cup V'_3]$, there is a dipath between $v$ and  $B^2_2$ contained in $D[T_0\cup B^2_2]$. By the assumption, there is a dipath $P$ from $B^2_2$ to $v$ contained in $D[T_0\cup B^2_2]$. Let $u\in V(P)\cap V(B^2_2)$.  Consider a shortest dipath $P_1$ in $D'$ from $B_1$ to $u$, the existence of $P_1$ can be guaranteed since $D'$ is strongly connected.  If $V(P_1)\cap B_2=\emptyset$, then
there is diwalk $P_1\cup P(u,v)$ from $B_1$ to $v$ without passing through $\{x\}\cup B_2$ which implies that $v\in B^1_1$, a contradiction.
Hence, we may assume that $V(P_1)\cap B_2\neq \emptyset$. Since $P_1$ is selected to be shortest,  there is a dipath $P'_1\subset P_1$ from $B_2$ to $u$ without passing through $\{x\}\cup B_1$  contradicting the fact that $u\in B^2_2$. Therefore, for each vertex $v$ in $T_0$ there is a dipath from $v$ to $B^2_2$ contained in $D[T_0\cup B^2_2]$.$\blacksquare$
\vspace{0.2cm}

By {\bf Claim 1}, we know that $\{\tilde{V}_1, \tilde{V}_2, \tilde{V}_3\}$ is a partition of $V(D)$.  Now, we prove that $D[\tilde{V}_i]$ is acyclic for $i\in \{1,2,3\}$.
Obviously, by the definition of $T_3$ and {\bf Claim 2}, $D[\tilde{V}_3]$ is acyclic. To get a contradiction, suppose that there is a dicycle $C_i$ contained in $D[\tilde{V}_i]$ for $i\in \{1,2\}$. For $i\in \{1,2\}$, we have that $D[V'_i]$ is acyclic by {\bf Claim 2}, and since  $D[T_1\cup E]$ and $D[T_2]$ are acyclic, we have that $V(C_i)\cap T_i\neq\emptyset$, $V(C_1)\cap (V'_1\setminus E)\neq\emptyset$  and $V(C_2)\cap V'_2\neq\emptyset$.
Let $v_i\in V(C_i)\cap T_i$, $u_1\in V(C_1)\cap (V'_1\setminus E)$ and $u_2\in V(C_2)\cap V'_2$ for $i\in \{1,2\}$.

We first consider the vertex set $\tilde{V}_1$. Suppose that $u_1\in B_2\cup B^1_2$. By {\bf Claim 3}, there is a dipath $P_1$ from $v_1$ to $B^2_2$ contained in $D[T_0\cup B^2_2]$. Let $w_1\in V(P_1)\cap B^2_2$. Therefore, by the definition of the partition,  there is a diwalk $W_1=C_1[u_1,v_1]\cup P_1$ without passing through $\{x\}\cup B_1$. The existence of the diwalk $W_1$ implies that $w_1\in B^1_2$, a contradiction.
Thus, $V(C_1)\cap (B_2\cup B^1_2)=\emptyset$ and so $u_1\in A_1$. By the definition of the partition, there is a dipath $P'_1$ from $B_1$ to $u_1$ without passing through $\{x\}\cup B_2$.  Thus, there is a diwalk $W'_2=P'_1\cup C_1[u_1,v_1]$ from $B_1$ to $v_1$ without passing through $\{x\}\cup B_2$. This implies that $v_1\in B^1_1$ contradicting the fact that $v_1\in B^2_1$. Therefore, $D[\tilde{V}_1]$ is acyclic.

Now, we consider the vertex set $\tilde{V}_2$. Suppose that $u_2\in B_1\cup B^1_1$, then there is a dipath containing $C_2[u_2,v_2]$ from $B_1$ to $v_2$ without passing through $\{x\}\cup B_2$. It implies that $v_2\in B^1_1$, a contradiction.
Thus, $V(C_2)\cap ( B_1\cup B^1_1)=\emptyset$ and so $u_2\in A_2$. Note that $v_2$ is contained in some dicycle $C'$ in $D[T_2\cup E]$. Let $w_2\in V(C')\cap E$. Then  $C'[v_2,w_2]$ a dipath  from $v_2$ to $w_2$ without passing through $\{x\}\cup B_1$.
Thus, there is a diwalk $W_2=C'[u_2,v_2]\cup C'[v_2,w_2]$ from $u_2$ to $w_2$ without passing through $\{x\}\cup B_1$. Combining that there is a dipath from $B_2$ to $u_2$ without passing through $\{x\}\cup B_1$, it implies that $w_2\in A_2$ contradicting the fact that $w_2\in E$. Therefore, $D[\tilde{V}_2]$ is acyclic.

Finally, we  conclude that $D$ admits an acyclic 3-coloring, which is a contradiction  to $\ve(D)=4$. So the assumption $\mc(K)>4$ is not true. Hence, $\mc(K)=4$. This completes the proof of Theorem~\ref{SH1=4}.
\end{proof}

\end{document}